\documentclass[12pt, reqno, microtype]{amsart}

\usepackage{amsmath, amssymb, amsthm}
\usepackage[hidelinks]{hyperref}
\usepackage[parfill]{parskip}
\usepackage{pinlabel}
\usepackage[all]{xy}
\usepackage[margin=1.18in]{geometry}
\usepackage{tikz}
\usetikzlibrary{matrix}

\usepackage{subfig}
\usepackage{mathrsfs}
\usepackage{mathtools}
\usepackage{mathabx}

\title{Cohomology of rotational tiling spaces}

\author{James J.\ Walton} 
\address{Department of Mathematics\\University of York\\\newline
         Heslington, York, YO10 5DD\\UK}
\email{jamie.walton@york.ac.uk}
\urladdr{http://maths.york.ac.uk/www/jjw548}

\thanks{Research supported by EPSRC}

\theoremstyle{definition}

\numberwithin{equation}{section}

\newcommand{\mf}{\mathfrak}
\newcommand{\mc}{\mathcal}

\newcommand{\R}{\mathbb{R}}

\newcommand{\Z}{\mathbb{Z}}
\newcommand{\N}{\mathbb{N}}

\DeclareMathOperator{\im}{\text{im}}

\newtheorem{theorem}{Theorem}[section]
\newtheorem*{theorem*}{Theorem}

\newtheorem*{lemma*}{Lemma}
\newtheorem{corollary}[theorem]{Corollary}
\newtheorem*{corollary*}{Corollary}

\theoremstyle{definition}

\newtheorem{definition}[theorem]{Definition}
\newtheorem*{definition*}{Definition}

\numberwithin{equation}{section} % Number equations by section
\numberwithin{figure}{section}   % Number figures by section

\graphicspath{ {images/} }

\begin{document}
\begin{abstract} A spectral sequence is defined which converges to the \v{C}ech cohomology of the Euclidean hull of a tiling of the plane with Euclidean finite local complexity. The terms of the second page are determined by the so-called ePE homology and ePE cohomology groups of the tiling, and the only potentially non-trivial boundary map has a simple combinatorial description in terms of its local patches. Using this spectral sequence, we compute the \v{C}ech cohomology of the Euclidean hull of the Penrose tilings.
\end{abstract}
\maketitle

\section*{Introduction}

To a tiling of Euclidean space one may associate a \emph{tiling space}, a moduli space of locally indistinguishable tilings equipped with a certain natural topology. In the field of aperiodic order, one is typically less interested in short-range features, such as the precise shapes of individual tiles of a given tiling, than long-scale features, such as the nature of recurrence of patches across the tiling. Passing from the original tiling to a tiling space (or associated dynamical system) is an elegant way of investigating an aperiodic tiling, since inessential features are forgotten whilst much useful qualitative information is retained. Recent work of Julien and Sadun makes this precise; to paraphrase, two translational hulls of FLC tilings being homeomorphic is equivalent to those tilings being MLD equivalent (local redecorations of each other) up to a possible `shape change' \cite{JulSad15}. For a general introduction to the study of aperiodic tilings through the topology of tiling spaces, see Sadun's book \cite{SadunBook} on the topic.

The \emph{translational hull} of a tiling $\mf{T}$ is the topological space $\overline{\mf{T}+\R^d}$, the completion of the translational orbit of $\mf{T}$ with respect to an intuitively defined metric on tilings. Whilst most attention in the literature has been given to the this choice of tiling space, there is another space naturally associated to $\mf{T}$, the \emph{Euclidean hull} $\Omega^\text{rot}$, which retains more information on the rotational symmetries of the tiling. It is given by taking the completion of the collection of tilings which are rigid motions of $\mf{T}$ (rather than just translations). For example, for $\mf{T}$ any Penrose kite and dart tiling $\Omega^{\text{rot}}$ is the moduli space of tilings of the plane which may be constructed from rigid motions of the Penrose kite and dart tiles, fitting together according to their matching rules.

A common topological invariant employed to study tiling spaces is \v{C}ech cohomology. For a two-dimensional substitution tiling, Barge, Diamond, Hunton and Sadun gave a spectral sequence converging to the \v{C}ech cohomology $\check{H}^\bullet(\Omega^{\text{rot}})$ of the Euclidean hull \cite{BDHS10}. The entries of the second page of this spectral sequence are determined by the number of tilings in the hull with non-trivial rotational symmetry (assumed all to be of the same order) and the \v{C}ech cohomology of the quotient space $\Omega^0 \coloneqq \Omega^{\text{rot}} / \text{SO}(2)$, given by identifying tilings of the Euclidean hull which agree up to a rotation at the origin. Unfortunately, there are often difficult extension problems left to resolve at the $E^\infty$ page of this spectral sequence.

In this paper we develop an alternative spectral sequence for computing the \v{C}ech cohomology of the Euclidean hull. It is applicable to any tiling with `Euclidean finite local complexity' (defined in Subsection \ref{subsec: Polytopal Tilings}), at least after passing to an equivalent polytopal tiling. The terms of this spectral sequence, given in Theorem \ref{thm: spectral sequence}, are determined by the `Euclidean pattern-equivariant' (ePE) cohomology and ePE homology groups of the tiling. The ePE homology, defined in \cite{Wal16}, is based on a similar construction to the well-known pattern-equivariant cohomology initially developed by Kellendonk and Putnam \cite{Kel03,KelPut06} and later reformulated in the cellular setting by Sadun \cite{Sad07}. Much like the spectral sequence of \cite{BDHS10}, the $E^2$ page is mostly determined by the \v{C}ech cohomology of $\Omega^0$ (which is isomorphic to the ePE cohomology), but the existence of rotationally invariant tilings of the hull often induces extra torsion in a single entry, here via the ePE homology term.

In Theorem \ref{thm: boundary} we show that there is a simple combinatorial description of the only potentially non-trivial boundary map of this spectral sequence, determined by the rigid-equivalence classes of `star-patches' about the vertices of the tiling. This makes the $E^\infty$ page computable for certain aperiodic tilings of interest. More seems to be resolved by the $E^\infty$ page than in the spectral sequence considered in \cite{BDHS10}. In particular, applied to the Penrose tilings all torsion is killed by the $E^\infty$ page, giving the following cohomology groups of its Euclidean hull:
\begin{align*}
\check{H}^k(\Omega^{\text{rot}}) = \begin{cases}
\Z   & \text{for } k=0;\\
\Z^2 & \text{for } k=1;\\
\Z^3 & \text{for } k=2;\\
\Z^2 & \text{for } k=3;\\
0    & \text{otherwise.}
\end{cases}
\end{align*}
This corrects the calculation of these cohomology groups published in \cite{BDHS10}. The result is surprising, since it seems to have been previously believed that in general the existence of pairs of rotationally invariant tilings in the hull inevitably leads to non-trivial torsion in the cohomology. One may expect for the exceptional fibres in the BDHS approximants of $\Omega^{\text{rot}}$ to produce non-trivial torsion elements in the degree one homology of the approximants, and so also torsion in the degree two cohomology. However, we show directly that for the Penrose kite and dart tilings the loop associated to the two rotationally invariant tilings of the hull is null-homotopic, explaining the lack of torsion in $\check{H}^2(\Omega^{\text{rot}})$.

The paper is organised as follows. In Section \ref{sec: Preliminaries} we recall the construction of the spaces $\Omega^{\text{rot}}$ and $\Omega^0$. We explain how their \v{C}ech cohomologies may be expressed in terms of pattern-equivariant cohomology, and recall the ePE homology defined in \cite{Wal16}. In Section \ref{sec: The Spectral Sequence and Boundary Map} we prove Theorem \ref{thm: spectral sequence}, on the existence of a spectral sequence whose $E^2$ page is given in terms of the ePE homology and ePE cohomology which converges to the \v{C}ech cohomology of $\Omega^{\text{rot}}$. We then define a special class of the ePE homology of a tiling which, according to Theorem \ref{thm: boundary}, determines the boundary map of the $E^2$ page of this spectral sequence. In Section \ref{sec: Cohomology of the Penrose Tilings} we apply our spectral sequence to the Penrose tilings. In the final section we present an alternative approach to these computations for tilings with translational finite local complexity. We apply this method to the Penrose tilings to obtain a completely different calculation of its cohomology groups which agrees with the approach via the ePE spectral sequence.

\subsection*{Acknowledgements}

The author thanks John Hunton and Dan Rust for helpful discussions related to this work.

\section{Preliminaries} \label{sec: Preliminaries}
\subsection{Polytopal Tilings} \label{subsec: Polytopal Tilings}

A \emph{polytopal tiling} of $\R^d$ is a pair $\mf{T} = (\mc{T},l)$, where $\mc{T}$ is a regular CW decomposition of $\R^d$ of polytopal cells and $l$ is a \emph{labelling}, a map from the set of cells to a set of `labels'. Frequently one has no need for a labelling, but equipping one can be helpful in case one wishes to allow for decorated patches so that two of which may be distinct despite being geometrically equivalent. Little generality is lost in considering only polytopal tilings, since any given tiling is always `S-MLD' equivalent to a polytopal tiling via a Voronoi construction. This equivalence relation was introduced in \cite{BSJ91}, and is the obvious extension to general rigid motions of the more standard `mutually locally derivable' (MLD) equivalence relation (of which see \cite{SadunBook}) in the translational setting. In particular, the topologies of the spaces $\Omega^0$ and $\Omega^\text{rot}$ (to be defined below) do not depend on the particular chosen representative of an S-MLD equivalence class.

A \emph{patch} of $\mf{T}$ is a finite subcomplex $\mc{P}$ of $\mc{T}$, equipped with the labelling $l$ restricted to $\mc{P}$. A patch consisting of only a single $d$-cell and its boundary cells shall be called a \emph{tile}. Two patches are \emph{rigid equivalent} if there is a \emph{rigid motion} (an orientation preserving isometry of $\R^d$) bijectively mapping the cells of one patch to the other in a way which respects the labelling. The \emph{diameter} of a patch is the diameter of the support of its cells. We say that $\mf{T}$ has \emph{Euclidean finite local complexity} (\emph{eFLC}) if, for any $r > 0$, there are only finitely many patches of diameter at most $r$ up to rigid equivalence. Since our tilings are always assumed to be polytopal, this is equivalent to asking for there to be only finitely many shapes of cells of $\mc{T}$ up to rigid motion, and that the labelling function takes on only finitely many distinct values.

Given a homeomorphism $\phi$ of $\R^d$, we define the tiling $\phi(\mf{T})$ in the obvious way, by applying $\phi$ to each cell of $\mf{T}$ and preserving labels (and similarly we may apply $\phi$ to finite patches). For $r \in \R_{>0}$ and cell $c$ of $\mf{T}$, we let the \emph{$r$-patch at $c$} be the patch of $\mf{T}$ supported on the set of tiles which are within radius $r$ of $c$. Denote by $E = \text{S}E(d) = \R^d \rtimes \text{SO}(d)$ the group of rigid motions of $\R^d$. For $k$-cells $c_1,c_2$ of $\mf{T}$, let $E_\mf{T}(c_1,c_2;r)$ be the set of rigid motions $\phi \in E$ which send $c_1$ to $c_2$ and the $r$-patch at $c_1$ to the $r$-patch at $c_2$. 

We shall always assume that $\mf{T}$ has \emph{trivial cell isotropy}: there exists some $r>0$ for which each $\phi \in E_\mf{T}(c,c;r)$ is the identity upon restriction to $c$. Any polytopal tiling can be made to have trivial cell isotropy by taking a barycentric subdivision \cite{Wal16}. For a cell $c$ of the tiling, let $\text{St}(c)$ be the \emph{star of $c$}, the patch of tiles $t$ incident with $c$, that is, with $\text{supp}(t) \supseteq c$. For example, the star $\text{St}(c)$ of a top-dimensional cell $c$ is simply a tile; for a tiling of $\R^2$, the star $\text{St}(e)$ of a $1$-cell $e$ is a two tile patch.

Up to an S-MLD equivalence given by a local redecoration of the tiles, we may strengthen the trivial cell isotropy condition with the following: for every $\phi \in E$ taking $c$ and $\text{St}(c)$ to themselves, $\phi$ is the identity map on $c$. It shall be convenient, and cause no loss of generality, to always assume this condition for eFLC tilings with trivial cell isotropy, to which our main theorems apply. In this case, we may \emph{consistently} orient each cell of $\mf{T}$, which is to say that whenever $\phi \in E$ takes $c_1$ to $c_2$ and $\text{St}(c_1)$ to $\text{St}(c_2)$, then $\phi$ maps our chosen orientation of $c_1$ to that of $c_2$; henceforth, we shall always assume that our cells are provided with orientations satisfying this condition. Assuming eFLC, there are finitely many distinct rigid equivalence classes of tiles; choose for each a representative \emph{prototile}. For a tiling of $\R^2$, assuming our strengthened trivial cell isotropy condition, for each rigid motion $t$ of prototile $p$ there exists a \emph{unique} rotation, denoted $\tau_t \in S^1$, for which $\tau_t (p)$ and $t$ agree up to translation. Let $|\tau_t| \in [0,2\pi)$ be the magnitude of rotation determining $\tau_t$, measured anticlockwise.

\subsection{Rotational Tiling Spaces}

A collection of tilings may be topologised by declaring that two tilings $\mf{T}_1$ and $\mf{T}_2$ are `close' whenever $\mf{T}_1$ and $\phi(\mf{T}_2)$ agree to a `large' radius about the origin for some `small' perturbation $\phi$, a homeomorphism of $\R^d$ which moves points within a large radius of the origin only a small amount. This is usually achieved via a metric; see \cite{SadunBook} or \cite{BDHS10}. There are inevitably some arbitrary and ultimately inconsequential choices to be made in the precise choice of the metric; an alternative (see \cite[\S1.2]{Wal16}) is to define a uniformity, which is essentially unique and has a very direct description from the notion of `closeness' described above.

\begin{definition} The \emph{Euclidean hull} of a tiling $\mf{T}$ is the topological space
\[
\Omega^\text{rot} = \Omega_\mf{T}^\text{rot} \coloneqq \overline{\{\phi(\mf{T}) \mid \phi \in \text{S}E(d)\}},
\]
the completion of the Euclidean orbit of $\mf{T}$ with respect to the metric discussed above.
\end{definition}

When $\mf{T}$ has eFLC, two tilings are `close' precisely when they agree about the origin to a `large' radius up to a `small' \emph{rigid motion}. Consequently it is not hard to show that $\Omega^\text{rot}$ is a compact space whose points may be identified with the set of tilings whose patches are all rigid motions of the patches of $\mf{T}$. The Euclidean group $E$ naturally acts on $\Omega^\text{rot}$ via $\phi \cdot \mf{T}' \mapsto \phi(\mf{T}')$, where $\phi \in E$ and $\mf{T}' \in \Omega^\text{rot}$. In particular, the group of rotations at the origin $\text{SO}(d) \leq E$ acts on the Euclidean hull. Following the notation of \cite{BDHS10}, we define $\Omega^0 \coloneqq \Omega^\text{rot} / \text{SO}(d)$, the quotient of the Euclidean hull which identifies tilings that are rotates of each other at the origin.

\subsection{Cohomology of Tiling Spaces} Our central goal is to calculate the \v{C}ech cohomology $\check{H}^\bullet(\Omega^{\text{rot}})$ of the Euclidean hull of an eFLC two-dimensional tiling. For certain tilings, in particular for substitution tilings, there are relatively simple methods of computing the cohomology $\check{H}^\bullet(\Omega^0)$ of the quotient \cite{BDHS10}. There is a close relationship between these cohomologies. For a two-dimensional tiling, the fibres $\pi^{-1}(x)$ of the quotient map $\pi \colon \Omega^{\text{rot}} \rightarrow \Omega^0$ may be identified with $\text{SO}(2) =: \text{S}^1$ at \emph{generic fibres}, corresponding to tilings with trivial rotational symmetry, and to $\text{SO}(2)/C_n =: \text{S}^1/n \cong \text{S}^1$ at \emph{exceptional fibres}, corresponding to tilings with $n$-fold rotational symmetry at the origin, where $C_n$ is the cyclic subgroup of rotations by $2 \pi k/n$ at the origin. Without these exceptional fibres, the map $\pi$ would be a fibration. In the presence of tilings with rotational symmetry, $\pi$ is no longer a fibration, although it is in some sense very close to one; in particular, one may show that over real coefficients $\check{H}^n(\Omega^{\text{rot}};\R) \cong \check{H}^n(\Omega^0 \times S^1;\R)$ (see \cite[Theorem 8]{BDHS10}). However, over integral coefficients the failure of $\pi$ to be a fibration at exceptional points can create extra torsion in $\check{H}^n(\Omega^{\text{rot}})$.

\subsection{Pattern-equivariant (co)homology for $\Omega^0$} Our approach shall rely on a highly geometric \emph{pattern-equivariant} point of view of these cohomology groups. Given a tiling $\mf{T} = (\mc{T},l)$, denote by $C^\bullet(\mc{T})$ the cellular cochain complex of $\mc{T}$. We shall say that a $k$-cochain $\psi \in C^k(\mc{T})$ is \emph{Euclidean pattern-equivariant} (\emph{ePE}) if there exists some $r >0$ for which, whenever $E_\mf{T}(c_1,c_2;r) \neq \emptyset$ for $k$-cells $c_1$ and $c_2$, then $\psi(c_1) = \psi(c_2)$ (recall that our cells are consistently oriented). In other words, to say that a cochain $\psi$ is ePE is simply to say that $\psi$ is constant on $k$-cells whose neighbourhood patches agree to a sufficiently large radius up to rigid motion. It is easy to show that the cellular coboundary of an ePE cochain is ePE, so we may define the sub-cochain complex $C^\bullet(\mf{T}^0)$ of $C^\bullet(\mc{T})$ consisting of ePE cochains. Its cohomology $H^\bullet(\mf{T}^0)$ is called the \emph{ePE cohomology of $\mf{T}$}.

Using a certain inverse limit presentation for $\Omega^0$ (see \cite{SadunBook}) and two fundamental properties of \v{C}ech cohomology, one may show the following:

\begin{theorem} Let $\mf{T}$ be an eFLC tiling with trivial cell isotropy. Then there exists a canonical isomorphism $\check{H}^\bullet(\Omega^0) \cong H^\bullet(\mf{T}^0)$. \end{theorem}

We define the \emph{ePE chain complex} $C_\bullet(\mf{T}^0)$ by replacing the cellular coboundary map by the cellular boundary map in $C^\bullet(\mf{T}^0)$; so we think of ePE chains as particular kinds of `infinite' cellular chains. The homology $H_\bullet(\mf{T}^0)$ is called the \emph{ePE homology of $\mf{T}$}. It was shown in \cite{Wal16} that over divisible coefficients the ePE homology groups are Poincar\'{e} dual to the ePE cohomology groups, which follows from a classical cell, dual-cell argument and a proof that the ePE homology is invariant under barycentric subdivision for divisible coefficients. However, these groups are not necessarily Poincar\'{e} dual over integral coefficients. For a two-dimensional tiling, one may modify the ePE chain complex $C_\bullet(\mf{T}^0)$ in degree zero so as to restore duality, as follows. Define the subgroup $C_0^\dagger(\mf{T}^0)$ of $C_0(\mf{T}^0)$ as the group of ePE chains $\sigma$ for which there exists some $r>0$ such that, whenever the patch of cells within radius $r$ of a vertex $v$ has $n$-fold symmetry about $v$, then $\sigma(v) = n \cdot k$ for some $k \in \Z$. In other degrees we set $C_i^\dagger(\mf{T}^0) \coloneqq C_i(\mf{T}^0)$. With the standard cellular boundary maps, this defines a sub-chain complex $C_\bullet^\dagger(\mf{T}^0)$ of the ePE chain complex.

\begin{theorem}[\cite{Wal16}] \label{thm: ePE} Let $\mf{T}$ be an eFLC tiling of $\R^d$ with trivial cell isotropy. Then
\begin{enumerate}
	\item if $G$ is a divisible coefficient ring with identity, $H^\bullet(\mf{T}^0;G) \cong H_{d-\bullet}(\mf{T}^0;G)$;
	\item if $d=2$, we have that $H^k(\mf{T}^0) \cong H_{2-k}(\mf{T}^0)$ for $k \neq 2$ and $H_0(\mf{T}^0)$ is an extension of $H^2(\mf{T}^0)$ by a torsion group whose generators have orders determined by the rotationally invariant tilings of $\Omega^0$;
	\item if $d=2$, we have that $H^\bullet(\mf{T}^0) \cong H_{2-\bullet}^\dagger(\mf{T}^0)$.
\end{enumerate}
\end{theorem}

\subsection{Pattern-equivariant (co)homology for $\Omega^\text{rot}$} The polytopal decomposition $\mc{T}$ of $\R^d$ defines a decomposition of $E$, by pulling back cells of $\mc{T}$ to $E$ via the fibration $q \colon E \rightarrow \R^d$ which sends $\phi \in E$ to $\phi(0) \in \R^d$. Denote by $\mc{T}^k$ the $k$-skeleton of $\mc{T}$ and $\mc{X}^k \coloneqq q^{-1}(\mc{T}^k)$. The decomposition $\mc{X}$ of $E$ is not cellular; of course, the preimage of a $k$-cell $c$ is homeomorphic to $\text{SO}(d) \times c$. To define the analogue of PE cohomology for $\Omega^\text{rot}$ we need to introduce additional cells so as to break up $\mc{X}$ into a cellular decomposition. In dimension $d=2$, a simple way of doing this is described in \cite{SadunBook}. For completeness we shall describe here a similar method.

For each tile $t$ of $\mf{T}$, we define two $3$-cells of $E$. Let $c_+^3(t)$ to be the set of $\phi \in E$ for which $\phi^{-1}(t)$ contains the origin in its interior and $|\tau_{\phi^{-1}(t)}| \in (0,\pi)$; that is, the tile corresponding to $t$ in $\phi^{-1}(\mf{T})$ lies over the origin, and is an anticlockwise rotate of its representative prototile. The cell $c_-^3(t)$ is defined analogously. The cells of $\mc{T}$ naturally define lower dimensional cells of $E$ in a similar manner which, together with the cells $c_\pm^3(t)$, define a CW decomposition $\mc{E}$ of $E$ which refines the non-cellular decomposition $\mc{X}$. For example, each $0$-cell of $\mc{E}$ is given by an element $\phi \in E$ for which $\phi^{-1}(\mf{T})$ has a vertex at the origin, incident with a tile $t$ for which $|\tau_t| \in \{0,\pi\}$. This breaks the fibre $q^{-1}(v) \cong S^1$ of a vertex $v$ of $\mf{T}$ into a union of $2n$ vertices and $2n$ open intervals, where $n$ is number of tiles incident with $v$ (at least typically---of course it may happen, and does not cause problems, for two tiles $t_1$ and $t_2$ incident with $v$ to satisfy $\tau_{t_1} = \pm \tau_{t_2}$). The cylinder $q^{-1}(e)$ of an open $1$-cell $e$ is (typically) decomposed into four $1$-cells and four $2$-cells, each region corresponding to when the two tiles incident with $e$ are oriented identically, oppositely, clockwise or anti-clockwise relative to their representative prototiles. Each solid torus $q^{-1}(t) \cong t \times S^1$ for a tile $t$ has its interior decomposed by the two $3$-cells $c_\pm^3(t)$ and the two $2$-cells corresponding to when $t$ is oriented identically, and oppositely to its representative prototile.

For $\phi \in E$ and $A \subseteq E$, let $\phi(A) \coloneqq \{\phi \circ a \mid a \in A\}$. Suppose that $c_1$, $c_2$ are cells of $\mf{T}$ for which $\phi(c_1)=c_2$. Then $\phi(q^{-1}(c_1)) = q^{-1}(c_2)$, so $\phi$ locally respects the decomposition $\mc{X}$. Moreover, for tiles $t_1$, $t_2$ with $\phi(t_1) = t_2$ we have that $\phi \circ \tau_{t_1}$ and $\tau_{t_2}$ agree up to a translation. So for $\phi \in E_\mf{T}(c_1,c_2;r)$, with $r$ sufficiently large, it is easy to see that $\phi(\mc{C}_1) = \mc{C}_2$ for subcomplexes $\mc{C}_i$ of $\mc{E}$ containing neighbourhoods of $q^{-1}(c_i)$; explicitly, we may take $\mc{C}_i$ to be the cells of $q^{-1}(\text{supp}(P))$, where $P$ is the patch of tiles contained in the $r$-patch at $c_i$ and not intersecting tiles of its complement. In other words, the elements of $E_\mf{T}(-,-;-)$ locally respect the CW decomposition $\mc{E}$.

A $k$-cochain $\psi \in C^k(\mc{E})$ is \emph{pattern-equivariant} (\emph{PE}) if there exists some $r \in \R_{>0}$ for which, for all $\phi \in E_\mf{T}(c_1,c_2;r)$ and every $k$-cell $c \in \mc{E}^k$ with $c \subset q^{-1}(c_1)$, we have that $\psi$ agrees on $c$ and $\phi(c)$. It is not hard to see that the coboundary of such a cochain is still PE, so we may define the sub-cochain complex $C^\bullet(\mf{T}^\text{rot})$ of PE cochains of $C^\bullet(\mc{E})$. Similarly, the \emph{boundary} of a PE cochain---thought of as an infinite chain---is still PE, so we may define the chain complex $C_\bullet(\mf{T}^\text{rot})$ by replacing the cellular coboundary maps of $C^\bullet(\mf{T}^\text{rot})$ with boundary maps.

\begin{theorem} Let $\mf{T}$ be a tiling of $\R^d$ with eFLC. There exist isomorphisms
\[
\check{H}^\bullet(\Omega^\text{rot}_\mf{T}) \cong H^\bullet(\mf{T}^\text{rot}) \cong H_{d-\bullet}(\mf{T}^\text{rot}).
\]
\end{theorem}

\begin{proof} The first isomorphism $\check{H}^\bullet(\Omega^\text{rot}_\mf{T}) \cong H^\bullet(\mf{T}^\text{rot})$ follows (see \cite{SadunBook}) from an analogous argument to the translational case (of which, see \cite{Sad07}). To recall some details, one may use the maps of $E_\mf{T}(c_1,c_2;r)$ to identify cells of $\mc{E}$ for successively larger $r$, constructing a diagram of CW complexes $\Gamma_i$, called \emph{approximants}, and quotient maps $f_i$ between them whose inverse limit is homeomorphic to $\Omega^\text{rot}$. The PE cochains of $\mc{E}$ are precisely the pullbacks of cochains from these approximants. Since the \v{C}ech cohomology of an inverse limit of compact CW complexes is naturally isomorphic to the direct limit of cellular cohomologies, it follows that
\[
\check{H}^\bullet(\Omega^\text{rot}) \cong \check{H}^\bullet(\varprojlim (\Gamma_i,f_i)) \cong \varinjlim (H^\bullet(\Gamma_i),f_i^*) \cong H(\varinjlim (C^\bullet(\Gamma_i),f_i^*)) \cong H^\bullet(\mf{T}^\text{rot}).
\]
For the second isomorphism $H^\bullet(\mf{T}^\text{rot}) \cong H_{d-\bullet}(\mf{T}^\text{rot})$, one may consider the CW decomposition $\mc{E}$ together with the maps $E_\mf{T}(-,-;-)$ as a `system of internal symmetries' with trivial isotropy, defined in \cite[\S3.5]{Wal16}. For such systems one may show that we have the desired Poincar\'{e} duality isomorphism, through a cell, dual-cell argument and by proving that these invariants are preserved under barycentric subdivision.\end{proof}

\section{The Spectral Sequence and Boundary Map} \label{sec: The Spectral Sequence and Boundary Map}

\begin{theorem} \label{thm: spectral sequence} Let $\mf{T}$ be a tiling of $\R^2$ with eFLC and trivial cell isotropy. Then there exists a spectral sequence (of homological type) converging to $H_\bullet(\mf{T}^\text{rot}) \cong \check{H}^{d-\bullet}(\Omega^\text{rot}_\mf{T})$ whose $E^2$ page is given by:
\begin{center}
\begin{tikzpicture}
  \matrix (m) [matrix of math nodes,
    		   nodes in empty cells,
    		   nodes={minimum width=2ex, minimum height=2ex,outer sep=-2pt},column sep=1ex,row sep=1ex,column 1/.style={nodes={text width=2ex,align=center}}]
    {
                &      &     &     & \\
          1     &  H_0(\mf{T}^0) &  \check{H}^1(\Omega^0)  & \check{H}^0(\Omega^0) & \\
          0     &  \check{H}^2(\Omega^0)  & \check{H}^1(\Omega^0) &  \check{H}^0(\Omega^0)  & \\
    \quad\strut &   0  &  1  &  2  & \strut \\};
  \draw[thick] (m-1-1.east) -- (m-4-1.east) ;
\draw[thick] (m-4-1.north) -- (m-4-5.north) ;
\end{tikzpicture}
\end{center}
\end{theorem}

\begin{proof} In analogy to the standard construction of the Serre spectral sequence for a fibration with CW base space, consider the filtration of the CW decomposition $\mc{E}$ by the subcomplexes supported on each $\mc{X}^k$ (recall that $\mc{X}^k$ is the preimage under the fibration $q \colon E \rightarrow \R^d$ of the $k$-skeleton $\mc{T}^k$). This induces a filtration of subcomplexes $C_\bullet(\mf{X}^k)$ of the PE chain complex $C_\bullet(\mf{T}^\text{rot})$ by restricting to those PE chains supported on each $\mc{X}^k$. Denote the relative chain complexes of successive levels of this filtration by $C_\bullet(\mf{X}^k,\mf{X}^{k-1})$. Associated to such a bounded filtration, there is bounded spectral sequence of homological type with
\[
E_{p,q}^1 = H_{p+q}(\mf{X}^p,\mf{X}^{p-1}) \Rightarrow H_{p+q}(\mf{X}^2) = H_{p+q}(\mf{T}^\text{rot}).
\]
We claim that there exist canonical chain isomorphisms
\begin{equation*}
E^1_{\bullet,0} = H_\bullet(\mf{X}^\bullet,\mf{X}^{\bullet-1}) \xrightarrow{\cong} C_\bullet^\dagger(\mf{T}^0); \
E^1_{\bullet,1} = H_{\bullet+1}(\mf{X}^\bullet,\mf{X}^{\bullet-1}) \xrightarrow{\cong} C_\bullet(\mf{T}^0),
\end{equation*}
with all other rows of the $E^1$ page trivial. The result then follows from the ePE Poincar\'{e} duality isomorphisms of Theorem \ref{thm: ePE}.

From our assumption of trivial isotropy on the cells of $\mf{T}$, we may consistently assign orientations to its cells with respect to the system of patch-preserving rigid motions of $E_\mf{T}(-,-;-)$. This naturally induces orientations on the cells of $\mc{E}$. Consider the map $F_0$ which sends an (infinite) cellular $k$-chain $\sigma$ of $\mc{E}$ to the $k$-chain of $\mc{T}$ defined on an open $k$-cell $c \in \mc{T}$ by
\[
F_0(\sigma)(c) \coloneqq \sum_{c' \in \mc{E}^k \text{; } c' \subset q^{-1}(c)} \sigma(c').
\]
It is easy to see that $F_0$ commutes with the boundary map. Moreover, it sends PE chains of $\mc{E}$ to ePE chains of $\mc{T}$. Indeed, let $\sigma$ be a PE $k$-chain of $\mc{E}$. For sufficiently large $r$, each $\phi \in E_\mf{T}(c_1,c_2;r)$ transports the chain at $q^{-1}(c_1)$ to that at $q^{-1}(c_2)$. Hence, the oriented sum of coefficients of $\sigma$ at $q^{-1}(c_1)$ and $q^{-1}(c_2)$ agree with respect to the map $\phi$, so $F_0(\sigma)$ is ePE for the same value of $r$.

It follows directly from the definition of $F_0$ that it does not depend on the value of a $k$-chain at cells of $\mc{X}^{k-1}$. We claim that it induces a chain isomorphism
\[
\tilde{F}_0 \colon H_\bullet(\mf{X}^\bullet,\mf{X}^{\bullet-1}) \xrightarrow{\cong} C^\dagger_\bullet(\mf{T}^0).
\]
Clearly outside of degrees $0,1,2$, both chain complexes are trivial. Let us firstly show that $\tilde{F}_0$ induces an isomorphism in degree two. For each tile $t$ of $\mf{T}$, there are precisely two $2$-cells $c_+^2(t)$ and $c_-^2(t)$ corresponding to $t$ in the relative complex $(\mc{X}^2,\mc{X}^1)$: the set of $\phi \in E$ for which the origin belongs to the interior of $\phi^{-1}(t)$ and $|\tau_{\phi^{-1}(t)}| = 0$ (that is, $\phi^{-1}(t)$ is oriented exactly as its associated prototile is) and, respectively, those $\phi$ for which the origin is interior to $\phi^{-1}(t)$ and $|\tau_{\phi^{-1}(t)}| = \pi$. Modulo PE boundaries, each element of $H_2(\mf{X}^2,\mf{X}^1)$ is uniquely represented by a chain concentrated on the cells of the form $c_+^2(t)$ and $F_0$ maps such elements isomorphically to $C_2(\mf{T}^0)$.

Showing that $\tilde{F}_0$ is an isomorphism in degree one is analogous: as before, by trivial cell-isotropy, we may pick a distinguished $1$-cell of $(\mc{X}^1,\mc{X}^0)$ for each $1$-cell of $\mf{T}$, consistently with respect to the maps of $E_\mf{T}(e_1,e_2;r)$; the argument proceeds as in degree two. Finally, we must show that $\tilde{F}_0$ induces an isomorphism in degree zero. Suppose that $\sigma \in C_0(\mf{X}^0)$. By pattern-equivariance, there exists some $r$ for which, for all $\phi \in E_\mf{T}(c_1,c_2;r)$, we have that $\phi$ maps the $0$-chain at $q^{-1}(c_1)$ to that at $q^{-1}(c_2)$. In particular, the chain $\sigma$ is preserved at any fibre $q^{-1}(v)$ by the maps of $E_\mf{T}(v,v;r)$. It follows that $F_0(\sigma)$ assigns a multiple of $\# E_\mf{T}(v,v;r)$ to $v$; that is, a multiple of the order of rotational symmetry of the $r$-patch at $v$. Hence, $F_0(\sigma) \in C^\dagger_0(\mf{T}^0)$. Conversely, by choosing distinguished $0$-cells in $\mc{X}^0$, given any $\sigma \in C^\dagger_0(\mf{T}^0)$ it is easy to see that we may lift it (uniquely, up to boundaries in the fibres) to a PE chain of $\mc{X}^0$ which maps to $\sigma$.

To show that $E^1_{\bullet,1} \cong C_\bullet(\mf{T}^0)$, note that each element of $E^1_{k,1} = H_{k+1}(\mf{X}^k,\mf{X}^{k-1})$ is defined by a chain of the form
\[
\sigma = \sum_{c \in \mc{T}^k} \sigma(c) \cdot \nu_c
\]
where $\sigma(c) \in \Z$ and $\nu_c$ is the sum of consistently oriented $(k+1)$-cells of $q^{-1}(c)$. To say that such a chain is PE is simply to say that the chain
\[
F_1(\sigma) \coloneqq \sum_{c \in \mc{T}^k} \sigma(c) \cdot c
\]
is ePE. It is easy to see that $\partial F_1(\sigma) = F_1(\partial \sigma)$, so this establishes the desired chain isomorphism. It is perhaps useful to point out that, in contrast to the isomorphism $H_0(\mf{X}^0) \cong C_0^\dagger(\mf{T}^0)$, there is no divisibility condition on the coefficients $\sigma(v)$ for vertices $v$ of rotational symmetry; the cycle $\nu_v$ is already invariant under rotation.

The above determines the first and second rows of the $E^1$ page of the spectral sequence. All cells of $\mc{X}^k$ not contained in $\mc{X}^{k-1}$ are either $k$ or $(k+1)$-dimensional, so the remaining entries $E^1_{p,q}$ for $q \neq 0,1$ are trivial. Applying homology and Theorem \ref{thm: ePE} completes the proof. \end{proof}

The $E^2$ page of the spectral sequence constructed above is concentrated in degrees $(p,q) \in \{0,1,2\} \times \{0,1\}$. As such, there is only one potentially non-trivial boundary map $\partial \colon E^2_{2,0} \rightarrow E^2_{0,1}$ required to determine the $E^\infty$ page. The $E^2_{2,0}$ entry is given by $\check{H}^0(\Omega^0)$, which is isomorphic to $\Z$ since $\Omega^0$ is connected, so we seek to describe $[\omega] = \partial(\Gamma)$ for a generator $\Gamma \in E^2_{2,0}$. It turns out that we may effectively describe such a representative $\omega$ in terms of the local combinatorics of the star-patches of $\mf{T}$.

Recall that each tile $t$ has associated to it a rotation $\tau_t$ relating it to its representative prototile, and each star of edge $\text{St}(e)$ determines an orientation for $e$. For each equivalence class of such star, pick a representative $\text{St}(e)$ and let $\rho_e \in \R$ be such that $\rho_e \equiv |\tau_{e(\text{l})}| - |\tau_{e(\text{r})}| \mod 2\pi$, where $e(\text{l})$ and $e(\text{r})$ are the tiles to the left and the right of $e$, respectively. Set $\rho_{e'} = \rho_{e}$ for each $e'$ whose star is rigid equivalent to that of $e$; clearly $\rho_{e'} \equiv |\tau_{e'(\text{l})}| - |\tau_{e'(\text{r})}| \mod 2\pi$ too. Repeat this process for each other class of edge.

Let $v$ be a vertex of $\mf{T}$ and list the edges $e_1, e_2, \ldots, e_k$ incident with $v$. Consider the sum $\sum_{i=1}^k \epsilon_i \rho_{e_i}$, where $\epsilon_i = 1$ if $e_i$ is oriented outwards from the vertex $v$, and $\epsilon_i = -1$ if it is oriented inwards. Passing to the quotient $S^1 = \R/2\pi \Z$, each entry $|\tau_t|$ appears twice with opposite signs in the sum, so $\sum_{i=1}^k \epsilon_i \rho_{e_i} \in 2\pi \Z$. Let $\omega(v) = \sum_{i=1}^k \rho_{e_i} / 2\pi \in \Z$; we may think of $\omega(v)$ as a winding number associated to the sequence of rotations $\rho_e$ relating the rotations $\tau_t$ of consecutive tiles incident with $v$. Repeating this procedure for each vertex of $\mf{T}$ defines an ePE $0$-chain $\omega \in C_0(\mf{T}^0)$.

\begin{theorem} \label{thm: boundary} Let $\partial \colon E_{2,0}^2 \rightarrow E_{0,1}^2$ be the boundary map of the $E^2$ page of the spectral sequence in Theorem \ref{thm: spectral sequence}. Then, with respect to the isomorphism $E_{0,1}^2 \cong H_0(\mf{T}^0)$, we have that $\partial(\Gamma) = [\omega]$ for a generator $\Gamma$ of $E_{2,0}^2 \cong \Z$, where $\omega \in C_0(\mf{T}^0)$ is the ePE $0$-chain constructed above. \end{theorem}

\begin{proof} To determine the image of $\partial$, it suffices to find a representative $\sigma \in C_2(\mf{X}^2)$ of the generator of $E^2_{2,0}$ and some $\tau \in C_2(\mf{X}^1)$ for which $\partial(\sigma + \tau) \in C_1(\mf{X}^0)$; then $\partial[\sigma]$ is precisely the homology class of $\partial(\sigma + \tau)$ in $E^2_{0,1} \cong H_0(\mf{T}^0)$.

Recall that the proof of Theorem \ref{thm: spectral sequence} relied on the chain isomorphisms $E^1_{\bullet,0} \cong C_\bullet^\dagger(\mf{T}^0)$ and $E^1_{\bullet,1} \cong C_\bullet(\mf{T}^0)$. A generator of $H_2^\dagger(\mf{T}^0)$ is represented by a fundamental class $\Gamma = \sum_{t \in \mc{T}^2} t$ of consistently oriented tiles in $\mf{T}$. Under the chain isomorphism $H_\bullet(\mf{X}^\bullet,\mf{X}^{\bullet-1}) \cong C_\bullet^\dagger(\mf{T}^0)$, this cycle is represented by the PE chain $\sigma \in C_2(\mf{T}^\text{rot})$ which assigns value $1$ to each cell $c_+^2(t)$ in $\mc{E}$ and zero to the others (recall the definition of the cell $c_+^2(t)$ from the proof of Theorem \ref{thm: spectral sequence}). Its boundary $\partial(\sigma)$ is supported on the $1$-cells $c^1(e,t)$ of $\mc{E}$ given by the set of $\phi \in E$ for which $\phi^{-1}(e)$ contains the origin and $|\tau_{\phi^{-1}(t)}| = 0$, where $e$ is a $1$-cell of $\mf{T}$ and $t$ is a tile incident with $e$. The two such edges at any cylinder $q^{-1}(e)$ occur with opposite orientations in $\partial(\sigma)$: positive for the edge $c^1(e,e(\text{r}))$ for the tile $e(\text{r})$ to the right of $e$ and negative for the edge $c^1(e,e(\text{l}))$ for the tile $e(\text{l})$ to the left of $e$ (without loss of generality, by choosing the appropriate orientation for the fundamental class $\Gamma$). Starting at the edge $c^1(e,e(\text{r}))$, rotate $\rho_e$ radians in the second coordinate about the cylinder $e \rtimes S^1$. This traces out a sum of rectangular $2$-cells in $q^{-1}(e)$, the lower edge of the first being $c^1(e,e(\text{r}))$ and the upper edge of the final one $c^1(e,e(\text{l}))$. Repeating this process for every other edge of the tiling defines a PE $2$-chain $\tau \in C_2(\mf{X}^1)$ which, by construction, is such that $\partial(\sigma + \tau) \in C_1(\mf{X}^0)$.

To determine $\partial(\sigma + \tau)$ note that, since $\partial(\sigma + \tau)$ is a cycle supported on the disjoint union $\mc{X}^0$ of circular fibres $q^{-1}(v)$, it suffices to calculate the sum of rotations at each fibre. Since $\partial(\sigma)$ was confined to cells not supported on $\mc{X}^0$, we need only calculate the sum for $\partial(\tau)$. At the fibre of a vertex $v$, the boundary $\partial(\tau)$ winds $\rho_e$ radians for each outwards oriented edge $e$, and $-\rho_e$ for each inwards pointing edge. This sum is precisely that used to define the chain $\omega$. \end{proof}

We recover as a corollary to the above theorems the following result, stated in \cite[Theorem 8]{BDHS10} for eFLC recognisable substitution tilings:

\begin{corollary} Let $\mf{T}$ be a tiling of $\R^2$ with eFLC. Then over real coefficients
\[
\check{H}^\bullet(\Omega_\mf{T}^\text{rot};\R) \cong \check{H}^\bullet(\Omega_\mf{T}^0 \times S^1;\R) \cong \check{H}^\bullet(\Omega_\mf{T}^0;\R) \oplus \check{H}^{\bullet-1}(\Omega_\mf{T}^0;\R).
\]
\end{corollary}

\begin{proof} We may assume that $\mf{T}$ has trivial cell isotropy (by taking a barycentric subdivision, which does not change the topology of the spaces $\Omega^0$ or $\Omega^\text{rot}$). Consider the real coefficient counterparts of theorems \ref{thm: spectral sequence} and \ref{thm: boundary}. The chain $\omega$ is directly constructed as a boundary in $C_\bullet(\mf{T}^0;\R)$ over real coefficients: it is the boundary of the ePE $1$-chain whose coefficient at an edge $e$ is given by $-\rho_e \in \R$. So the boundary map at the $E^2$ page is trivial and $E^2 \cong E^\infty$. Over real coefficients, $E_{p,0}^\infty \cong \check{H}^{2-p}(\Omega^0;\R)$ and, since $H_0(\mf{T}^0;\R) \cong \check{H}^2(\Omega^0;\R)$ by Theorem \ref{thm: ePE}, $E_{p,1}^\infty \cong \check{H}^{2-p}(\Omega^0;\R)$ also; all other rows are trivial. There are no extension problems over field coefficients, and so the result follows. \end{proof}

\section{Cohomology of the Penrose Tilings} \label{sec: Cohomology of the Penrose Tilings}
\subsection{Invariants of Hierarchical Tilings}
Many interesting aperiodic tilings have the special property of being equipped with a hierarchical structure, typically described in terms of a `substitution rule'. For tilings such as this, Anderson and Putnam \cite{AndPut98} showed how one may construct a CW approximant $\Gamma$ with self-map $f$ for which the translational hull $\Omega^1 \coloneqq \overline{\mf{T}+\R^d}$ is homeomorphic to the inverse limit of the diagram $\Gamma \xleftarrow{f} \Gamma \xleftarrow{f} \cdots$. This allows the \v{C}ech cohomology of $\Omega^1$ to be computed. Barge, Diamond, Hunton and Sadun showed how one may construct similar such approximants for $\Omega^1$ (which often leads to simpler computations than the method of \cite{AndPut98}), as well as for the rotational tiling spaces $\Omega^0$ and $\Omega^\text{rot}$. In \cite{Wal16}, we showed how one may compute $H_\bullet(\mf{T}^0)$ of a hierarchical tiling with eFLC, the method being close in spirit to the computations of \cite{BDHS10} for invariants of $\Omega^0$. For $\mf{T}$ a tiling of $\R^2$, the method is easily modified to compute the \v{C}ech cohomology $\check{H}^\bullet(\Omega^0) \cong H_{2-\bullet}^\dagger(\mf{T}^0)$. The approach is phrased directly in terms of the star-patches of the tiling, and so dovetails conveniently with the description of the $E^\infty$ page of the spectral sequence provided by Theorem \ref{thm: boundary}.

\subsection{The Penrose Tilings}
Let $\mf{T}$ be a tiling of Penrose's famous `kite' and `dart' tiles which meet along edges according to their matching rules \cite{Pen84}. These tilings have a hierarchical structure, based upon self-similarity with inflation constant the golden ratio. Applying the computations of \cite{Wal16} to $\mf{T}$, one calculates:
\begin{align*}
H_k(\mf{T}^0) = \begin{cases}
	\Z^2 \oplus \Z / 5 & \text{for } k=0;\\
	\Z                 & \text{for } k=1;\\
	\Z                 & \text{for } k=2;\\
	0                  & \text{otherwise.}
\end{cases}
& & &
H_k^\dagger(\mf{T}^0) = \begin{cases}
	\Z^2 & \text{for } k=0;\\
	\Z   & \text{for } k=1;\\
	\Z   & \text{for } k=2;\\
	0    & \text{otherwise.}
\end{cases}
\end{align*}
These ePE homology groups determine the $E^2$ page of the spectral sequence of Theorem \ref{thm: spectral sequence}; see Figure \ref{fig: Penrose sp seq}

\begin{figure}
	\begin{center} 
		\begin{tikzpicture}
  \matrix (m) [matrix of math nodes,
    		   nodes in empty cells,
    		   nodes={minimum width=2ex, minimum height=2ex,outer sep=-2pt}, column sep=1ex,row sep=1ex, column 1/.style={nodes={text width=2ex,align=center}}]
    {
                &      &     &     & \\
          1     &  \Z^2 \oplus \Z / 5 &  \Z  & \Z & \\
          0     &  \Z^2  & \Z &  \Z  & \\
    \quad\strut &   0  &  1  &  2  & \strut \\};
\draw[thick] (m-1-1.east) -- (m-4-1.east) ;
\draw[thick] (m-4-1.north) -- (m-4-5.north) ;
\draw[-stealth] (m-3-4.north west) -- (m-2-2.south east);
		\end{tikzpicture}
\hspace{2cm}
		\begin{tikzpicture}
  \matrix (m) [matrix of math nodes,
    		   nodes in empty cells,
    		   nodes={minimum width=2ex, minimum height=2ex,outer sep=-2pt}, column sep=1ex,row sep=1ex, column 1/.style={nodes={text width=2ex,align=center}}]
    {
                &      &     &     & \\
             1  & \Z^2 &  \Z & \Z  & \\
             0  & \Z^2   & \Z  &  \Z & \\
    \quad\strut & 0    &  1  &  2  & \strut \\};
\draw[thick] (m-1-1.east) -- (m-4-1.east) ;
\draw[thick] (m-4-1.north) -- (m-4-5.north) ;
		\end{tikzpicture}
	\end{center}
\caption{The $E^2$ (left) and $E^\infty$ page (right) of the spectral sequence for the Penrose tilings.}
\label{fig: Penrose sp seq}
\end{figure}

\begin{figure}
    \centering
    \captionsetup{width=0.45\linewidth}
    \subfloat[The tiles of the Penrose tilings are given rotational orientations in the plane (left). These define the values $\rho_e$ (mod $2\pi$) of the $1$-stars (middle). The chain $\omega$ is defined by the resulting winding numbers at each $0$-star (right).]{%
        \includegraphics[width=0.5\textwidth]{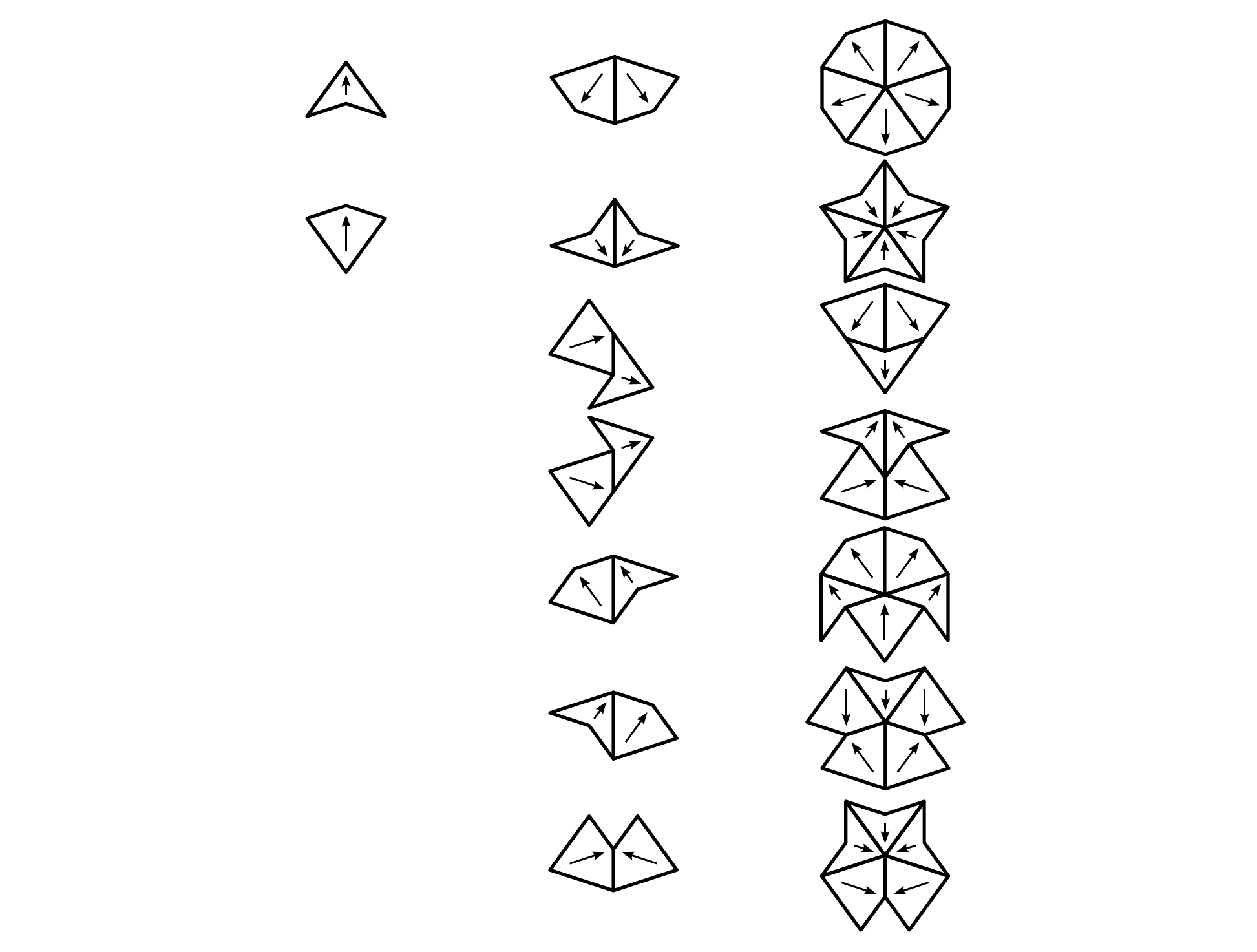}%
        \label{fig: star-patches}%
        }%
    \hfill%
    \subfloat[Illustration of homotopies $a \simeq p \ast a$ (bottom triangle) and $p^{-1} \ast a \simeq b$ (top triangle).]{%
        \includegraphics[width=0.5\textwidth]{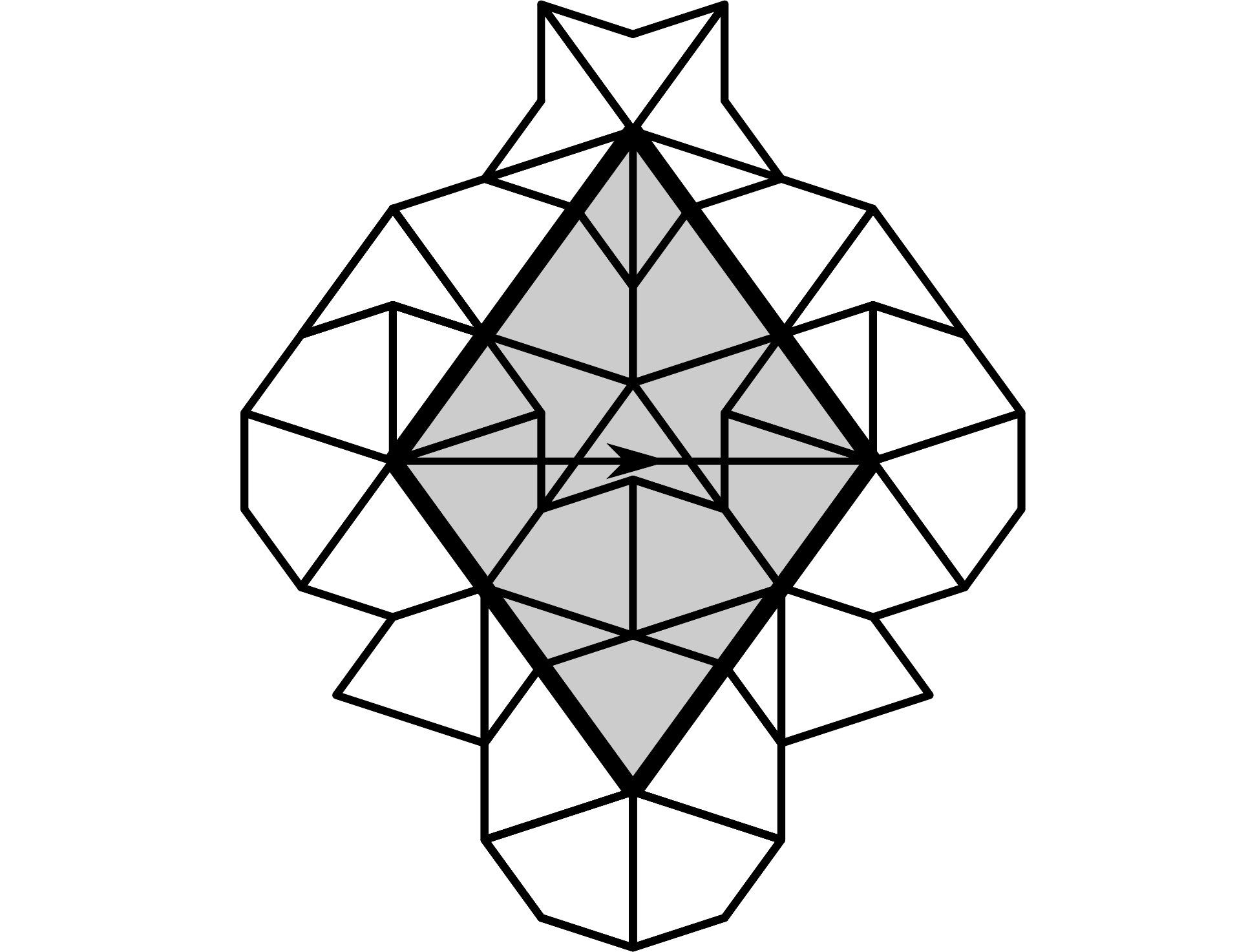}%
        \label{fig: homotopies}%
        }%
       \caption{}
\end{figure}

Theorem \ref{thm: boundary} determines the differential $\partial \colon E_{2,0}^2 \rightarrow E_{0,1}^2$. There are two rigid equivalence classes of tiles in a Penrose tiling: kites and darts. They have official rotational orientations in the plane, as indicated by the arrows of Figure \ref{fig: star-patches}. For the star of an edge $e$, recall that we must set $\rho_e$ so that $\rho_e \equiv |\tau_{e(\text{l})}| - |\tau_{e(\text{r})}| \mod 2\pi$. The value $\rho_e$ is simply the angle of rotation which takes the arrow of the right tile to that of the left. We let $\rho_e$ be the rotation of smallest such magnitude, counted positive for anticlockwise rotations and negative clockwise rotations. Then $\rho_{e_k} = 2\pi n_k /10$, where $n_k=-2$, $+2$, $+1$, $-1$, $0$, $0$ and $-4$, for the seven edge-stars, with respect to their ordering in Figure \ref{fig: star-patches}. To calculate $\omega$, we sum the values $\rho_e/2\pi$ about each vertex, taken with the appropriate sign according to whether $e$ is oriented inwards or outwards from $v$; equivalently, calculate the winding number of the motion that the arrow takes about a vertex, rotating from one tile to the next by the smallest magnitude. Then $\omega(v_k) = +1$, $+1$, $0$, $0$, $0$, $-1$, $0$ for the seven vertex-stars, ordered as in Figure \ref{fig: star-patches}.

The ePE $0$-chain $\omega$ is precisely the representative of a generator of the $5$-torsion of $H_0(\mf{T}^0)$ calculated in \cite{Wal16}. So the $5$-torsion at the $E^2$ page is killed. This leaves the $E^\infty$ page free Abelian (see Figure \ref{fig: Penrose sp seq}) and so there are no extension problems. It follows that $\check{H}^k(\Omega^\text{rot}) \cong \Z,\Z^2,\Z^3,\Z^2$ for $k=0,1,2,3$, and is trivial in other degrees.

\subsection{Exceptional Fibres}

A common explanation for additional torsion in $\check{H}^2(\Omega^\text{rot})$ is that exceptional fibres in the BDHS approximants $\Gamma$ to $\Omega^\text{rot}$ (see \cite{BDHS10}), corresponding to patches of rotational symmetry, lead to non-trivial torsion in $H_1(\Gamma)$. The result above shows that this cannot happen for the Penrose tilings, we shall demonstrate that directly here.

The action of rotation naturally induces an action on $\Gamma$. Let $r \colon \Gamma \rightarrow \Gamma/S^1 =: \Gamma^0$ be the quotient; the space $\Gamma^0$ is an approximant for $\Omega^0$. The $S^1$-action is free on all but finitely many isolated fibres $r^{-1}(x)$, corresponding to patches with non-trivial rotational symmetry. This is the analogue, from $3$-manifold theory, of a `Seifert fibred space' (only an analogue since the approximant $\Gamma$ here is a \emph{branched} manifold). We call the fibres upon which $S^1$ acts freely \emph{generic}; the fibres upon which a cyclic subgroup $C_n \leq S^1$ of order $n$ acts as the identity shall be called \emph{$n$-exceptional}.

Given $x \in \Gamma^0$, there is a corresponding loop $l_x$ about the fibre of $r^{-1}(x)$ (traversing it, say, in an anticlockwise direction). The loops $l_x$ and $l_y$ at generic fibres are always homotopic; indeed, connect $x$ to $y$ by a continuous path and shift $l_x$ to $l_y$ fibre-wise. We may not shift the loop $l_x$ of an $n$-exceptional point $x$ to a generic fibre $y$ in this fashion; instead, we have that $l_y$ is homotopic to $n \cdot l_x$. If we have two $n$-exceptional points $x,y \in \Gamma^0$, we may attempt to create $n$-torsion in $\pi_1(\Gamma)$, and hence $n$-torsion in $H^2(\Gamma)$, by considering the concatenation of loops $l_x \ast l_y^{-1}$ (where we concatenate as standard by connecting the loops to some arbitrary base point). Since they may not be simply shifted onto nearby generic fibres, we may suspect that $l_x$ and $l_y$ are not homotopic, but we have that $n \cdot l_x \simeq l_z \simeq n \cdot l_y$ for a generic point $z$, so $n(l_x \ast l_y^{-1}) = 0$ and $l_x \ast l_y^{-1}$ is $n$-torsion. Unfortunately, this argument relies on $l_x$ and $l_y$ not being homotopic; they are certainly not naively homotopic by shifting the loops fibre-wise. However, they may be homotopic via a more subtle procedure.

Consider the Penrose tiling spaces $\Omega^\text{rot}$ and $\Omega^0$, with corresponding BDHS approximants $\Gamma \coloneqq K_\epsilon^\text{rot}$ and $\Gamma^0 \coloneqq K_\epsilon^0$ for small $\epsilon > 0$ (see the notation of \cite{BDHS10} for the BDHS approximants $K_\epsilon^0$ and $K_\epsilon^\text{rot}$). There are two $5$-exceptional points $x$ and $y$ of $\Gamma^0$, corresponding to the centres of the `sun' and `star' patches, respectively (which are the first and second vertex-stars of Figure \ref{fig: star-patches}). Our calculation of the cohomology $\check{H}^\bullet(\Omega^\text{rot})$ implies that the exceptional loops $a \coloneqq l_x$ and $b: = l_y$ represent the same homology classes. We shall in fact show that they are homotopic loops.

The patch of tiles of Figure \ref{fig: homotopies} is a valid patch from a Penrose tiling, so a path of it staying clear (with respect to sufficiently small $\epsilon$) of its boundary projects to $\Gamma^0$. A further continuous choice of rotational orientation of this patch lifts the path to a path of $\Gamma$. In this manner, the loop $a$ is represented by staying at the sun vertex at the bottom of the patch, and continuously rotating the patch anti-clockwise from its original orientation to $2\pi/5$ by the end of the motion. The loop $b$ is similarly represented, by staying put at the upper star vertex, and rotating the picture continuously by $2\pi/5$. Consider the pair of horizontally aligned sun vertices at the centre of the patch. We may define a loop $p$ in $\Gamma$ by setting $p(t)$ to be the point of $\Gamma$ corresponding to the patch at $p'(t)$ for a continuous path $p'$ travelling horizontally rightwards from the left sun vertex to the right one. We claim that $p$ is null-homotopic in $\Gamma$, and that $b \simeq p^{-1} \ast a$. It then follows that $a = b$ in $\pi_1(\Gamma)$.

Both of these claims follow from a quick examination of Figure \ref{fig: homotopies}. Consider the left (resp.\ right) line segment between the lower sun vertex and the left (resp.\ right) sun vertex of the patch. Let $x_s$ and $y_s$ be the pair of points lying on the left and right line segments, respectively, at vertical displacement $s$ between the lower sun vertex and the two upper ones (normalised so that this vertical displacement is $1$); so $x_0 = y_0$ corresponds to the bottom sun vertex and $x_1$, $y_1$ correspond to the left and right sun vertices, respectively. For $t \in [0,1/2]$, we define $l_s(t)$ to be the point of $\Gamma$ corresponding to the point of the patch $(1-2t) \cdot x_s + 2t \cdot y_s$; that is, $l_s(t)$ travels linearly from $x_s$ at time $t=0$ to $y_s$ at time $t=1/2$. At times in $[1/2,1]$ we define $l_s(t)$ by linearly rotating the patch at $l_s(1/2)$ anticlockwise, so that $l_s(1)$ is eventually rotated $2\pi/5$ relative to $l_s(1/2)$. Note that the patches at $x_s$ and $y_s$ are related by a rotation by $2\pi/5$ at $y_s$, so each $l_s$ defines a loop. Since $l_s(t)$ varies continuously in $s$ and $t$, we have defined a homotopy between $l_0 \simeq a$ and $l_1 = p \ast a$. So $a = p \ast a$ in $\pi_1(\Gamma)$ which implies that $p$ is null-homotopic.

A second homotopy of loops $l'_s$ is defined similarly. This time, we begin with the loop $l_0' = p^{-1} \ast a$, visualised by travelling from right to left between the two central sun vertices, and then rotating at the end of the loop by $2\pi/5$. Analogously to above, we may continuously shift this loop upwards, so that $l'_1 \simeq b$. It follows that $p^{-1} \ast a = b$ in $\pi_1(\Gamma)$. Since $p$ is null-homotopic, it follows that the exceptional loops $a$ and $b$ are homotopic. This agrees with our calculation of $\check{H}^\bullet(\Omega^\text{rot})$, the loop $a \ast b^{-1}$ is null-homotopic and does not induce torsion in $\check{H}^2(\Omega^\text{rot})$.

\section{Tilings of Finite Local Complexity} \label{sec: Tilings of Finite Local Complexity}

A tiling is said to have \emph{translational finite local complexity} (\emph{FLC}) if there are only finitely many patches of diameter at most $r$ up to translation equivalence for each $r>0$. Many interesting examples of aperiodic tilings (such as the Penrose tilings) have FLC, although some only satisfy the weaker condition of having eFLC (such as the Conway--Radin pinwheel tilings). The spectral sequence developed above not only gives a method of computing the \v{C}ech cohomology of an eFLC tiling for which the ePE (co)homology may be computed, it also provides a filtration of its cohomology in terms of these ePE invariants. For a tiling with FLC, we shall provide an alternative approach to computation of its cohomology which produces a different decomposition in terms of invariants of its translational hull.

\subsection{Translational Hulls}
For a tiling $\mf{T}$ of $\R^d$ with FLC, its \emph{translational hull} is the topological space
\[
\Omega^1 = \Omega^1_\mf{T} \coloneqq \overline{\mf{T} + \R^d},
\]
the completion of the translational orbit of $\mf{T}$ with respect to the standard tiling metric. Say that a finite subgroup $\Theta \leq \text{SO}(d)$ \emph{acts on $\mf{T}$ by rotations} if, for every patch $P$ of $\mf{T}$ and rotation $\phi \in \Theta$, we have that $\phi(P)$ is also a patch of $\mf{T}$, up to translation. Note that, in this case, $\Theta$ naturally acts on $\Omega^1$ via $\phi \cdot \mf{T}' \mapsto \phi(\mf{T}')$ for $\phi \in \Theta$ and $\mf{T}' \in \Omega^1$. If, in addition, we have that patches $P$ and $Q$ of $\mf{T}$ agree up to rigid motion if and only if $P$ and $\phi(Q)$ agree up to translation for some $\phi \in \Theta$, then we say that $\mf{T}$ \emph{has rotation group $\Theta$}. One may easily construct tilings which do not have rotation groups, but they tend to be somewhat artificial.

\subsection{Mapping Tori}
Let $X$ be a compact, Hausdorff space and $f \colon X \rightarrow X$ be a homeomorphism. The \emph{mapping torus of $f$} is defined to be the quotient space
\[
X_f \coloneqq \frac{X \times [0,1]}{(x,1) \sim (f(x),0)}.
\]
There is a map $\pi \colon X_f \rightarrow \R/ \Z$ defined by setting $\pi(x,t) = [t]$, making $X_f$ a fibre bundle over $S^1$ with fibres $X$. Cut $S^1$ into two closed semicircles and consider their preimages $U,V$; each is homotopy equivalent to $X$ and their intersection $U \cap V$ is homotopy equivalent to a disjoint union of two copies of $X$. The associated Mayer--Vietoris sequence reads
\[
\cdots \check{H}^k(X_f) \rightarrow \check{H}^n(X) \oplus \check{H}^k(X) \xrightarrow{\Psi} \check{H}^k(X) \oplus \check{H}^k(X) \rightarrow \check{H}^{k+1}(X_f) \rightarrow \cdots,
\]
where $\Psi$ is given by $\Psi(x,y) = (x+y,x+f^*(y))$; see \cite[\S2]{Hun15}. With some further simple algebraic manipulations, we may express the cohomology of $X_f$ in terms of invariants and coinvariants of $f$. We have short-exact sequences
\[
0 \rightarrow \text{coinvar}^{k-1}(f) \rightarrow \check{H}^k(X_f) \rightarrow \text{invar}^k(f) \rightarrow 0
\]
where
\begin{equation*}
	\begin{aligned}
\text{invar}^k(f^*) \coloneqq \ker(\operatorname{id}-f^* \colon \check{H}^k(X) \rightarrow \check{H}^k(X)); \\
\text{coinvar}^k(f^*) \coloneqq \frac{\check{H}^k(X)}{\im(\operatorname{id}-f^* \colon \check{H}^k(X) \rightarrow \check{H}^k(X))}.
	\end{aligned}
\end{equation*}

\subsection{Rotational Tiling Spaces as Mapping Tori}

Let $\mf{T}$ be a tiling of $\R^2$ with FLC and rotation group $\Theta$. Then $\Theta = C_n$ for some $n \in \N$, where $C_n$ is the cyclic subgroup of $\text{SO}(2)$ of rotations by $2\pi k/n$. Rotation by $2\pi/n$ generates $\Theta$, and induces a homeomorphism $f \colon \Omega^1 \rightarrow \Omega^1$. It is not difficult to see that $\Omega^\text{rot}$ is homeomorphic to the mapping torus of $f$; a point of the mapping torus represented by $(\mf{T}',t) \in \Omega^1 \times [0,1]$ is identified with $\phi_{2\pi t/n}(\mf{T}')$, where $\phi_s \in \text{SO}(2)$ is the rotation by $s$ at the origin. Then from our discussion above on the cohomologies of mapping tori, we obtain the following:

\begin{theorem} For a tiling $\mf{T}$ of $\R^2$ with FLC and rotation group $\Theta$ of order $n$, we have short exact sequences
\[
0 \rightarrow \text{coinvar}^{k-1}(f^*) \rightarrow \check{H}^k(\Omega^\text{rot}) \rightarrow \text{invar}^k(f^*) \rightarrow 0
\]
for all $k \in \N_0$, where $f \colon \Omega^1 \rightarrow \Omega^1$ is the homeomorphism induced by the rotation by $2\pi /n$.
\end{theorem}

\subsection{Invariants and Coinvariants of the Penrose Tilings}

The translational hull $\Omega^1$ of the Penrose tilings has cohomology $\check{H}^k(\Omega^1) \cong \Z,\Z^5,\Z^8$ in degrees $k=0,1,2$, respectively. The Penrose tilings have rotation group of order $10$; let $f$ denote the homeomorphism of $\Omega^1$ induced by rotation by $2\pi/10$. Since each cohomology group of $\Omega^1$ is free Abelian, so is $\text{invar}^k(f^*)$ for each $k \in \N_0$, so there are no extension problems and the cohomology of the Euclidean hull splits as a direct summand $\check{H}^k(\Omega^\text{rot}) \cong \text{coinvar}^{k-1}(f^*) \oplus \text{invar}^k(f^*)$ of the invariants and coinvariants.

These cohomology groups, and the induced action $f^*$ on them, may be computed using the methods of \cite{AndPut98, BDHS10}, or via PE homology as described in \cite{Wal16}. Using the latter approach, we directly compute
\[
\text{invar}^k(f^*) \cong \Z, \Z, \Z^2; \  \text{coinvar}^k(f^*) \cong \Z, \Z, \Z^2
\]
for $k=0,1,2$, respectively. This agrees with our computation via the ePE spectral sequence. For example, in the degree $k=2$ of particular interest where extra torsion is potentially picked up, we compute that $f^*$ acts on $H^1(\Omega^1) \cong \Z^5$ by the matrix
\[
M \coloneqq 
\left(
\begin{smallmatrix}
1 & 0 & 0 & 0 & 0 \\
0 & 0 & 0 & 0 & -1 \\
0 & 1 & 0 & 0 &  1 \\
0 & 0 & 1 & 0 & -1 \\
0 & 0 & 0 & 1 &  1 \\
\end{smallmatrix}
\right)
\]	

Then $M - \operatorname{id}$ sends the first basis vector to zero and is an isomorphism upon restriction to the subspace spanned by the remaining four basis vectors. It follows that $\text{coinvar}^1(f^*) \cong \Z$ and we may already conclude that $\check{H}^2(\Omega^\text{rot}) \cong \Z \oplus \text{invar}^2(f^*)$ is free Abelian. We may similarly compute $\text{invar}^2(f^*)$ directly, or note that the rank of $\text{invar}^2(f^*)$ is equal to the rank of $\check{H}^2(\Omega^0) \cong \Z^2$ (c.f., \cite[Theorem 7]{BDHS10}, \cite[Proposition 3.12]{Wal16}). It follows that $\check{H}^2(\Omega^\text{rot}) \cong \Z^3$.

\bibliographystyle{abbrv}

\bibliography{./biblio}

\begin{thebibliography}{10}

\bibitem{AndPut98}
J.~E. Anderson and I.~F. Putnam.
\newblock Topological invariants for substitution tilings and their associated
  {$C\sp *$}-algebras.
\newblock {\em Ergodic Theory Dynam. Systems}, 18(3):509--537, 1998.

\bibitem{BSJ91}
M.~Baake, M.~Schlottmann, and P.~D. Jarvis.
\newblock Quasiperiodic tilings with tenfold symmetry and equivalence with
  respect to local derivability.
\newblock {\em J. Phys. A}, 24(19):4637--4654, 1991.

\bibitem{BDHS10}
M.~Barge, B.~Diamond, J.~Hunton, and L.~Sadun.
\newblock Cohomology of substitution tiling spaces.
\newblock {\em Ergodic Theory Dynam. Systems}, 30(6):1607--1627, 2010.

\bibitem{Hun15}
J.~Hunton.
\newblock Spaces of projection method patterns and their cohomology.
\newblock In {\em Mathematics of aperiodic order}, volume 309 of {\em Progr.
  Math.}, pages 105--135. Birkh\"auser/Springer, Basel, 2015.

\bibitem{JulSad15}
A.~{Julien} and L.~{Sadun}.
\newblock {Tiling deformations, cohomology, and orbit equivalence of tiling
  spaces}.
\newblock {\em ArXiv e-prints}, June 2015.

\bibitem{Kel03}
J.~Kellendonk.
\newblock Pattern-equivariant functions and cohomology.
\newblock {\em J. Phys. A}, 36(21):5765--5772, 2003.

\bibitem{KelPut06}
J.~Kellendonk and I.~F. Putnam.
\newblock The {R}uelle-{S}ullivan map for actions of {$\Bbb R\sp n$}.
\newblock {\em Math. Ann.}, 334(3):693--711, 2006.

\bibitem{Pen84}
R.~Penrose.
\newblock Pentaplexity: a class of nonperiodic tilings of the plane.
\newblock In {\em Geometrical combinatorics ({M}ilton {K}eynes, 1984)}, volume
  114 of {\em Res. Notes in Math.}, pages 55--65. Pitman, Boston, MA, 1984.

\bibitem{Sad07}
L.~Sadun.
\newblock Pattern-equivariant cohomology with integer coefficients.
\newblock {\em Ergodic Theory Dynam. Systems}, 27(6):1991--1998, 2007.

\bibitem{SadunBook}
L.~Sadun.
\newblock {\em Topology of tiling spaces}, volume~46 of {\em University Lecture
  Series}.
\newblock American Mathematical Society, Providence, RI, 2008.

\bibitem{Wal16}
J.~J. {Walton}.
\newblock {Pattern-Equivariant Homology}.
\newblock {\em ArXiv e-prints}, Jan. 2014.

\end{thebibliography}

\end{document}